\documentclass[11pt]{article}

\usepackage[pagebackref,colorlinks=true,urlcolor=blue,linkcolor=red,citecolor=red]{hyperref}

\usepackage{epsfig,graphics, graphicx}
\usepackage{subcaption, comment, bm}
\usepackage[percent]{overpic}
\usepackage{float}
\usepackage{amssymb,bm}
\usepackage{amsthm}
\usepackage{color}
\usepackage[]{amsmath}
\usepackage[]{amsfonts}
\usepackage[]{fancyhdr}
\usepackage[]{graphicx,wrapfig}
\usepackage{enumitem}
\usepackage[utf8]{inputenc}
\usepackage[T1]{fontenc}
\usepackage{mathtools}
\usepackage{array}
 \usepackage{pdfsync,verbatim,stmaryrd}
 \newcommand{\Beq}{\begin{equation}}
 \newcommand{\Eeq}{\end{equation}}
 \newcommand{\beq}{\begin{equation*}}
 \newcommand{\eeq}{\end{equation*}}
 \newcommand{\bal}{\begin{align}}
 \newcommand{\eal}{\end{align}}

 \newcommand{\Rc}{\mathcal{R}}
 
 \newcommand{\Tc}{\mathcal{T}}

 \newcommand{\Wc}{\mathcal{W}}

 \newcommand{\Rb}{\mathbb{R}}
 \newcommand{\Sb}{\mathbb{S}}

 \renewcommand{\o}{\omega}
 \usepackage{amsfonts}
 \newcommand{\FR}{\mathbb{R}} 
 
 \DeclareMathAlphabet{\bi}{OML}{cmm}{b}{it}
 \DeclareMathAlphabet{\bcal}{OMS}{cmsy}{b}{n}
 \DeclareMathAlphabet{\brmn}{OT1}{cmr}{bx}{n}
 \usepackage{amsopn}
 \usepackage{amsmath,amsthm,amssymb}

\newtheorem{remark}{Remark}
\newtheorem{proposition}{Proposition}
\newtheorem{exmp}{Example}

 \newtheorem{theorem}{Theorem}
 
 \newtheorem{lemma}{Lemma}

 \newtheorem{definition}{Definition}

\title{\vspace{-1cm} Inversion of a restricted transverse ray transform with sources on a curve}

\author{Rohit Kumar Mishra\thanks{Department of Mathematics, Indian Institute of Technology, Gandhinagar, Gujarat, India. \url{rohit.m@iitgn.ac.in}, \url{rohittifr2011@gmail.com}} \and Chandni Thakkar\thanks{Department of Mathematics, Indian Institute of Technology, Gandhinagar, Gujarat, India. \url{thakkar_chandni@iitgn.ac.in}}}

\begin{document}
\maketitle
\begin{abstract}
 In this paper, a restricted transverse ray transform acting on vector and symmetric $m$-tensor fields is studied. We developed inversion algorithms using restricted transverse ray transform data to recover symmetric $m$-tensor fields in $\FR^3$ and vector fields in $\FR^n$. We restrict the transverse ray transform to all lines going through a fixed curve $\gamma$ that satisfies the Kirillov-Tuy condition. We show that the known restricted data can be used to reconstruct a specific weighted Radon transform of the unknown vector/tensor field's components, which we then use to explicitly recover the unknown field. 
 \end{abstract}
\section{Introduction}\label{Sec:Introduction}
The question of inversion of various types of integral transforms, such as longitudinal ray transform, momentum ray transforms, and transverse ray transform (TRT), has attracted a lot of attention due to their wide range of applications in various imaging fields. The problem at hand is the reconstruction of a symmetric $m$-tensor field from the knowledge of its integral transforms. In this work, we solve the following two inversion problems using the transverse ray transform data:
\begin{itemize}
\item Reconstruction of a symmetric $m$-tensor field from the knowledge of its TRT (for lines intersecting a fixed curve $\gamma$ satisfying the Kirillov-Tuy condition) in $\mathbb{R}^3$
\item Reconstruction of a vector field from the knowledge of its TRT (for lines intersecting a fixed curve $\gamma$ satisfying the Kirillov-Tuy condition) in $\mathbb{R}^n$
\end{itemize}
The transverse ray transform was introduced by Vladimir Sharafutdinov \cite{Sharafutdinov_1994} to study questions related to polarization. In recent years, TRT has appeared in various fields, such as in recovering strain field for polycrystalline materials from the X-ray diffraction data \cite{Lionheart_Withers_2015}, to approximate the change in the polarization state of light from truncated TRT \cite{Sharafutdinov_Lionheart_2009}, in Histogram tomography \cite{Lionheart_2020}, and many other areas \cite{Griesmaier_TRT_2018, Holman_2009, Holman_2013, Sharafutdinov_Novikov_2007,  Sharafutdinov_1994}. Let us start with some of the necessary notations before we define the transverse ray transform and state the main results of the article. 

Let $\mathit{S}^m_n = \mathit{S}^m(\mathbb{R}^n)$ denote the space of covariant symmetric $m$-tensor fields in $\mathbb{R}^n$ and let $C_c^\infty(\mathit{S}^m_n)$ denote the space of smooth compactly supported symmetric $m$-tensor fields in $\mathbb{R}^n$. In local coordinates, $f \in C_c^\infty(\mathit{S}^m_n)$ can be expressed as 
\begin{equation*}
    f(x) = f_{i_1\dots i_m}(x) \,dx^{i_1} \dots \,dx^{i_m},
\end{equation*}
where $f_{i_1\dots i_m}(x)$ are compactly supported smooth functions which are symmetric in their indices. Throughout this article, we assume Einstein summation convention for repeated indices. The space of straight lines in $\mathbb{R}^n$ is  parametrized by 
\[T\mathbb{S}^{n - 1} = \left\{(\xi, x) \in \mathbb{S}^{n - 1} \times \mathbb{R}^n : \left<x, \xi\right> = 0\right\} = \big\{(\xi, x) \in \mathbb{R}^{n} \times \mathbb{R}^n : \left<x, \xi\right> = 0, \ |\xi| = 1\big\} .\]
Given a $( \xi,x) \in T\mathbb{S}^{n - 1}$, we can find a unique straight line $\left\{ x + t \xi : t \in \mathbb{R}\right\}$ in $\mathbb{R}^n$ passing through $x$ and in the direction $\xi$. Further, let us denote by
\[T\mathbb{S}^{n - 1} \oplus T\mathbb{S}^{n - 1} = \left\{(\xi, x, y) \in \mathbb{S}^{n - 1} \times \mathbb{R}^n \times \mathbb{R}^n : \left<x, \xi\right> = 0, \left<y, \xi\right> = 0\right\}.\]
the Whitney sum of $T\mathbb{S}^{n - 1}$ with itself. Here $\langle \cdot, \cdot \rangle$ denotes the usual dot product in $\mathbb{R}^n$.

\begin{definition}[\cite{Microlocal_2021,Sharafutdinov_1994}]\label{def:TRT}
    The transverse ray transform $\mathcal{T} : C_c^\infty(S^m_n) \rightarrow C^\infty(T\mathbb{S}^{n - 1} \oplus T\mathbb{S}^{n - 1})$ is the bounded linear map defined by 
    \[\mathcal{T}f (\xi, x, y) = \int_\mathbb{R} \left< f (x + t \xi), y^{\odot m} \right> \,dt,\]
    where $y^{\odot m}$ denotes the $m^{th}$ symmetric tensor product of $y$ and $\left< f(x), y^{\odot m} \right> = f_{i_1\dots i_m}(x) y^{i_1} \cdots y^{i_m}$.
\end{definition}\
To get a more explicit and workable definition of TRT, we use the spherical coordinates to parameterize the direction of integration.
To achieve this, we start with a brief discussion about the spherical coordinates in $\FR^n$ (also known as hyperspherical coordinates). For every $\xi \in \mathbb{S}^{n - 1}$, there exists $\varphi_1, \varphi_2, \dots, \varphi_{n - 2} \in [0, \pi]$ and $\varphi_{n - 1} \in [0, 2\pi)$ such that $\xi$ can be represented as follows:
\begin{align*}
    \xi &= \begin{bmatrix}
            \sin\varphi_1 \sin\varphi_2 \dots \sin\varphi_{n - 2} \sin\varphi_{n - 1}\\
            \sin\varphi_1 \sin\varphi_2 \dots \sin\varphi_{n - 2} \cos\varphi_{n - 1}\\
            \vdots \\
            \sin\varphi_1 \cos\varphi_2\\
            \cos\varphi_1
        \end{bmatrix}.
\end{align*}
Consider the orthonormal frame $\displaystyle \left\{\xi, \eta_1, \dots, \eta_{n-1} \right\}$ in $\Rb^n$ with  $\eta_j$ (for $j = 1, \dots , n-1$) defined as follows:
\begin{align}\label{eq:orthogonal frame}
    \eta_{1} &= \begin{bmatrix}
        \cos\varphi_1 \sin\varphi_2 \dots \sin\varphi_{n - 2} \sin\varphi_{n - 1}\\
        \cos\varphi_1 \sin\varphi_2 \dots \sin\varphi_{n - 2} \cos\varphi_{n - 1}\\
        \vdots \\
        \cos\varphi_1 \cos\varphi_2 \\
        -\sin\varphi_1
    \end{bmatrix}, \nonumber\\
    \eta_{2} &= \begin{bmatrix}
        \cos\varphi_2 \dots \sin\varphi_{n - 2} \sin\varphi_{n - 1}\\
        \cos\varphi_2 \dots \sin\varphi_{n - 2} \cos\varphi_{n - 1}\\
        \vdots \\
        -\sin\varphi_2\\
        0
    \end{bmatrix},\nonumber \\
    &\qquad \vdots\nonumber\\
    \eta_{n - 1} &= \begin{bmatrix}
        \cos\varphi_{n - 1}\\
        -\sin\varphi_{n - 1}\\
        0\\
        \vdots\\
        0\\
        0
    \end{bmatrix}.
\end{align}
For notational convenience and calculation simplifications, we will use $\eta_{1} = \xi_\alpha$ and $\eta_{2} = \xi_\beta$, that is, we use  $\displaystyle \left\{\xi,  \xi_{\alpha}, \xi_{\beta} \right\}$ as an orthonormal frame for $\Rb^3$. Now, we define the vectorial versions of the TRT for 3-dimensional symmetric $m$-tensor fields and $n$-dimensional vector fields separately. It is known that this vectorial definition of TRT is equivalent (see \cite{Microlocal_2021}) to the  Definition \ref{def:TRT} defined above.
\begin{definition}[\cite{Microlocal_2021}]\label{def:vector version of TRT}
For $0 \leq i \leq m,$ define $\mathcal{T} = (\mathcal{T}_i) : C_c^\infty(S^m_3) \rightarrow (C^\infty(T\mathbb{S}^2))^{m + 1}$ by:
\begin{equation} \label{eq:vector version of TRT}
\mathcal{T}_if(x, \xi) = \int_\mathbb{R} \left<f(x + t \xi), \xi_\alpha^{\odot i} \odot \xi_\beta^{\odot (m-i)}\right> \,dt
\end{equation}
where $\odot$ denotes the symmetrized tensor product and $\xi_\alpha$ and $\xi_\beta$ denote the orthonormal vectors associated to $\xi$ as defined above.
\end{definition}
\begin{remark}
From the definition of $\Tc_i$, we observe 
$$\Tc_i f(x, \xi) =  (-1)^m \Tc_i f(x, -\xi),$$
which means $\Tc_i$ is an even/odd function in the $\xi$ variable if the order of the tensor is even/odd.
\end{remark}
\begin{definition}\label{def:vector version of TRT n-D}
For $1 \leq i \leq n - 1,$ define $\mathcal{T} = (\mathcal{T}_i) : C_c^\infty(S^1_n) \rightarrow (C^\infty(T\mathbb{S}^{n - 1}))^{n - 1}$ by:
\begin{equation} \label{eq:vector version of TRT n-D}
\mathcal{T}_if(x, \xi) = \int_\mathbb{R} \left<f(x + t \xi), \eta_{i}\right> \,dt = \int_\mathbb{R} f_j(x + t \xi) \eta_{i}^j \,dt
\end{equation}
where $\eta_{i}$'s are defined above and are orthogonal to $\xi$.
\end{definition}
Note that we are using the same symbol for TRTs defined above. It will be clear from the context which definition we are using depending on whether $f$ is a vector field in $\Rb^n$ or $f$ is a symmetric $m$-tensor field in $\Rb^3$.

In simple words, TRT integrates the components of a tensor field that are orthogonal to the line of integration $\left\{ x + t \xi: t \in \mathbb{R}\right\}$. The problem of interest here is to recover the unknown field from the knowledge of its TRT. The reconstruction problem, under consideration for TRT and for integral transforms over straight lines, is overdetermined for $n \geq 3$ since the space of lines in $\mathbb{R}^n$ is $(2n - 2)$-dimensional, and we are trying to recover a $n$-dimensional object. The natural question is then to ask if the reconstruction is possible with specific $n$-dimensional restricted data. Such restricted data problems for scalar functions are well studied by many authors; for instance, see \cite{Scalar_restricted_1986, Scalar_restricted_1980, Grangeat_scalar_incomplete_data, Uhlmann_Greenleaf_1989, Kirillov_1961,  Palamodov_scalar_incomplete_data, Tuy_1983}. The problems for vector fields and higher-order tensor fields are also considered in various settings, see \cite{Denisjuk_2006, Denisjuk_1994, Derevtsov2011,derevtsov_Svetov, Katsevich2007, Katsevich2013, Katsevich_Schuster, Microlocal_2018, Lan_thesis_restricted_transform_1999, Louis_2022, Rohit_2020, Mukhometov_1985, Schuster2000,  Schuster2001} and references therein. In this work, we assume the TRT is given for lines intersecting a fixed curve $\gamma$ in $\mathbb{R}^n$, which corresponds to $n$-dimensional data. To the best of authors' knowledge, there are only two works available with restricted TRT; the first one is \cite{Lionheart_Naeem_2016}, where authors have developed a slice-by-slice filtered back projection reconstruction algorithm for symmetric 2-tensor fields in $\mathbb{R}^3$ while the second one  \cite{Microlocal_2021} deals with microlocal inversion for symmetric $m$-tensor fields in $\mathbb{R}^3$. The articles \cite{Anuj_TRT_Support, UCP_2022, Holman_2013,  TRT_injectivity} address the injectivity questions for TRT in different settings.  For more literature related to transverse ray transform, please refer to \cite{Griesmaier_TRT_2018, Joonas_2023, Lionheart_Withers_2015, Sharafutdinov_Lionheart_2009,  Sharafutdinov_Novikov_2007, Sharafutdinov_Polarization}, and the references therein. 

The article is organized as follows. In Section \ref{Sec:Notations}, we describe the Kirillov-Tuy condition with an example and state the two main results of this article. Sections \ref{sec:proof of th 1} and \ref{sec:proof of th 2} are devoted to proving the two main theorems stated in section \ref{Sec:Notations}. Finally, we conclude this paper with an appendix in Section \ref{sec:appendix}.
and the acknowledgments in Section \ref{sec:acknowledge}.

\section{Some notations and statement of main theorems}\label{Sec:Notations}
The fixed curve $\gamma$, where we restrict the lines to collect the TRT data, satisfies a specific condition known as the Kirillov-Tuy condition. This condition on the curve was introduced for the scalar case in \cite{Kirillov_1961,Tuy_1983} and later it was defined for the higher order tensor fields in \cite{Denisjuk_2006,Vertgeim_2000}. We start this section by defining \textit{generic vectors}, which are important for defining the Kirillov-Tuy condition.
\begin{definition}[Generic vectors \cite{Denisjuk_2006}]
    Let us say that the vectors $e_1, \dots, e_{\nu(m, n)} \in \FR^{n - 1}$, with $\nu(m,n) = \binom{m + n - 2}{m}$, are generic, if any symmetric tensor $f$ in $\FR^{n - 1}$ of order $m$ is uniquely defined by the values $\left<f, e_1^{\odot m}\right>, \dots, \left<f, {e_{\nu(m, n)}^{\odot m}}\right>$.
\end{definition}

\begin{definition}[Kirillov-Tuy condition]
Fix a domain $B \subset \mathbb{R}^n$. We say a curve $\gamma$ satisfies the Kirillov-Tuy condition of order $m$ if for almost any hyperplane $H_{\omega,p} = \{x \in \mathbb{R}^n : \left<\omega, x\right> = p\}$ intersecting the domain $B$, there is a set of points $\gamma_1,\dots,\gamma_{\nu(m,n)} \in H_{\omega,p} \cap \gamma$, which locally smoothly depends on $(\omega, p)$, such that for almost every $x \in H_{\omega,p} \cap B$, the vectors $x - \gamma_1,\dots,x - \gamma_{\nu(m,n)}$ are generic.
\end{definition}
\begin{exmp}\cite{Vertgeim_2000}
Consider $B = B(0, r) \subset \mathbb{R}^3$ and $\gamma$ is a union of 3 orthogonal great circles on the sphere $S(0, R)$ with $ R > \sqrt{3}r.$ Then every hyperplane, passing through $B$, will intersect $\gamma$ at least two different  points $\gamma_1 \ \& \ \gamma_2$ and for almost every $x$ in the hyperplane, the vectors $x - \gamma_1 \ \& \ x -\gamma_2$ are linearly independent. Hence $\gamma$ satisfies the Kirillov-Tuy condition of order 1.
\end{exmp}
\noindent For our second objective of recovering a vector field in $\FR^n$ from the restricted TRT data, a weaker version of the Kirillov-Tuy condition is required, which is defined below:
\begin{definition}[Modified Kirillov-Tuy condition]
Fix a domain $B \subset \mathbb{R}^n$. We say a curve $\gamma$ satisfies the modified Kirillov-Tuy condition if for almost any hyperplane $H_{\omega,p} = \{x \in \mathbb{R}^n : \left<\omega, x\right> = p\}$ intersecting the domain $B$, there are points $\gamma_1, \gamma_2 \in H_{\omega,p} \cap \gamma$, which locally smoothly depends on $(\omega, p)$, such that for almost every $x \in H_{\omega,p} \cap B$, the vectors $x - \gamma_1,$ and $x - \gamma_2$ are linearly independent.
\end{definition}
\begin{definition}[\cite{Vertgeim_2000}]
A curve $\gamma$ is said to encompass a bounded domain $B$, if $\gamma \cap B = \emptyset$ and for each $a \in \gamma$ and $\xi \in \mathbb{R}^n\symbol{92} \left\{0\right\}$, at most one of the rays $\{a + t\xi : t \geq 0\}$ and $\{a + t\xi : t \leq 0\}$ intersects $B$.
\end{definition}
\begin{remark}
Let $\gamma$ be a curve that encompasses the support of a symmetric $m$-tensor field $f$. Then $\mathcal{T}_if$ takes the following form, if $\{a + t\xi : t \geq 0\}$ intersects $B$, 
\[\mathcal{T}_if (a, \xi) = \mathcal{T}_i^+ f(a, \xi) = \int_0^\infty \left<f(a + t \xi), \xi_\alpha^{\odot i} \odot \xi_\beta^{\odot (m-i)}\right> \,dt. \]
\noindent Otherwise, if $ \{a + t\xi : t \leq 0\} = \{a - t\xi : t \geq 0\}$ intersects $B$, we get the following expression for $\mathcal{T}_if$:
\begin{align*}
\mathcal{T}_i^- f(a, \xi) &= \int_0^\infty \left<f(a - t \xi), \xi_\alpha^{\odot i} \odot \xi_\beta^{\odot (m-i)}\right> \,dt\\
&= {(-1)}^m \int_0^\infty \left<f(a + t (-\xi)), (-\xi)_\alpha^{\odot i} \odot (-\xi)_\beta^{\odot (m-i)}\right> \,dt\\
&= {(-1)}^m \mathcal{T}_i^+ f(a, -\xi).
\end{align*}
From this relation, we see that knowing $\displaystyle \left\{\mathcal{T}_i^+ f (a, \xi): a \in \gamma \mbox{ and } \xi\in\Sb^{n-1}\right\}$ is same as knowing $\left\{\mathcal{T}_i^- f (a, \xi): a \in \gamma \mbox{ and } \xi\in\Sb^{n-1}\right\}$. Hence, for the rest of the paper, we will only work with $\mathcal{T}_i^+ f$ and call it $\mathcal{T}_i f$. To summarise, for $a \in \gamma$ and $\xi \in \mathbb{S}^{n - 1}$, we will use the following expression for $\mathcal{T}_if$: \[\mathcal{T}_i f(a, \xi) = \int_0^\infty \left<f(a + t \xi), \xi_\alpha^{\odot i} \odot \xi_\beta^{\odot (m-i)}\right> \,dt.\] 

\noindent A similar change in the expression appears for the vector field case. 
\end{remark}

\begin{remark}
The TRT defined in definition \ref{def:vector version of TRT} (definition \ref{def:vector version of TRT n-D}) can be extended to the space $\mathbb{R}^3 \times \mathbb{R}^3\setminus \{0\}$ ($\mathbb{R}^n \times \mathbb{R}^n\setminus \{0\}$), where we remove the restriction on $\xi$ to be a unit vector. Let us denote this extended operator by $\widetilde{\mathcal{T}}_if$. The operators $\widetilde{\mathcal{T}}_if$ and $\mathcal{T}_if$ are equivalent and are related in the following way :
\[\widetilde{\mathcal{T}}_if (x, \xi) = |\xi|^{m - 1} \mathcal{T}_if \left(x, \frac{\xi}{|\xi|}\right).\]
This extension allows us to define the partial derivatives $\frac{\partial}{\partial \xi_i}$. The above expression shows that knowing one of the operators gives all the information about the other operator. Hence, we will not distinguish between $\widetilde{\mathcal{T}}_if$ and $\mathcal{T}_if$, and just use the notation $\mathcal{T}_if$ in the further discussion.
\end{remark}

\noindent Now, we are ready to state the main theorems of this article.
\begin{theorem} \label{th:Main Theorem}
    Let $f$ be a symmetric $m$-tensor field in $\mathbb{R}^3$, which is supported in $B \subset \mathbb{R}^3$. Assume that a curve $\gamma \subset \mathbb{R}^3$, satisfying the Kirillov-Tuy condition of order $m$, encompasses $B$ and the transverse ray transform $\mathcal{T}_i f$, for $ i = 0, \dots, m$ is known for all lines intersecting the curve $\gamma$. Then the tensor field $f$ can be uniquely reconstructed in terms of $\Tc_i f$.
\end{theorem}
\begin{theorem} \label{th:Main Theorem 2}
    Let $f \in C_c^\infty(S^1_n)$, supported in $B \subset \mathbb{R}^n$. Assume that a curve $\gamma \subset \mathbb{R}^n$, satisfying the modified Kirillov-Tuy condition, encompasses $B$ and the transverse ray transform $\mathcal{T}_i f$, for $ i = 1, \dots, n - 1$ is known for all lines intersecting the curve $\gamma$. Then the vector field $f$ can be uniquely reconstructed in terms of $\Tc_i f$.
\end{theorem}
The main idea here is to relate the given transverse ray transform data to the Radon transform of $\left<f, \theta^{\odot m}\right>, \mbox{ for any arbitrary }\theta \in \mathbb{R}^n \backslash \left\{0\right\}$, then use the Radon inversion formula to recover the action $\left<f, \theta^{\odot m}\right>, \mbox{ for any arbitrary }\theta \in \mathbb{R}^n \backslash \left\{0\right\}$, which can be used to recover $f$ explicitly. We end this section with a very brief discussion of the Radon transform and some of its properties.

Let $\mathcal{S}(\mathbb{R}^n)$ denote the Schwartz space on $\mathbb{R}^n$. Then for $f \in \mathcal{S}(\mathbb{R}^n)$, the Radon transform is given by the following formula:
\[f^\wedge(\omega,p) = \Rc f(\o,p) = \int_{H_{\omega,p}}f(x)\,ds\]
where $ds$ is the standard volume element on hyperplane $H_{\omega,p}$. And the inversion formula for the Radon transform is given by (see \cite[Theorem 3.6]{Helgason_1999}):
\[f(x) = \left(\mathcal{R}^{-1}f^\wedge\right) (x) = \begin{cases*}
 \frac{1}{2} \frac{1}{{(2 \pi i)}^{n - 1}} \int_{\mathbb{S}^{n - 1}} \frac{\partial^{n-1}}{\partial p^{n-1}} f^\wedge(\omega, x\cdot \omega)\,d\omega & if  $n \text{ is odd}$  \\
 \frac{1}{2 \pi i} \int \int_{\FR \times \mathbb{S}^{n - 1}} \frac{1}{p - x\cdot \omega} \frac{\partial^{n-1}}{\partial p^{n-1}} f^\wedge(\omega, p)\,d\omega \,dp & if $n \text{ is even}$
\end{cases*}\]
where the integral with respect to $p$ (for even $n$) is considered as the Cauchy principal value; for details, please refer \cite[Lemma 3.7]{Helgason_1999}.

Also, the following property of the Radon transform is very useful for our analysis:
\begin{equation} \label{Radon transform property}
    \left(a_i\frac{\partial}{\partial x^i}f(x)\right)^\wedge(\omega,p) = \left<\omega,a\right> \frac{\partial}{\partial p} f^\wedge(\omega,p).
\end{equation}
The upcoming two sections are devoted to proving the Theorem \ref{th:Main Theorem} and \ref{th:Main Theorem 2} respectively. Each section will present some lemmas and propositions needed to prove the respective theorem. 
\section{Proof of Theorem \ref{th:Main Theorem}}\label{sec:proof of th 1}
We start by proving the following proposition, which will play a crucial role in completing the proof of Theorem \ref{th:Main Theorem}. 
\begin{proposition} \label{prop: Main proposition}
For $0 \leq i \leq m$, assume $\Tc_i f$ is known for all lines intersecting the curve $\gamma$ satisfying the conditions of Theorem \ref{th:Main Theorem}. Then, for any $\theta \in \mathbb{R}^3 \backslash \left\{0\right\}$, $\left<f, \theta^{\odot m}\right>$  can be written explicitly in terms of $\Tc_i f$.
\end{proposition}

\noindent The proof of this proposition is large and hence divided into small lemmas for better readability. 
\begin{lemma}
    Let $\xi_1, \xi_2, \dots, \xi_{m + 1}$ be a collection of pairwise linearly independent unit vectors in some hyperplane in $\mathbb{R}^3$. Then the following collection of $\binom {m + 2}m$-tensors is linearly independent:
    \[ \left\{ \left\{{(\xi_j)}_\alpha^{\odot {m}}\right\}_{j = 1}^{m + 1}, \left\{{(\xi_j)}_\alpha^{\odot {m - 1}} \odot {(\xi_j)}_\beta\right\}_{j = 1}^m, \dots, \left\{{(\xi_j)}_\alpha \odot {(\xi_j)}_\beta^{\odot {m - 1}} \right\}_{j = 1}^2, {(\xi_1)}_\beta^{\odot {m}}\right\}. \]
\end{lemma}
\begin{proof}
This lemma has been proved in a slightly different setup in \cite[Lemma 3.4]{Microlocal_2021}. In fact, Remark 3.1 of \cite{Microlocal_2021} comments about the setup we are considering. We prefer to give a brief sketch (main steps) of the proof for the sake of completeness.  The main steps of the proof are as follows (details can be found in \cite[Lemma 3.4]{Microlocal_2021}):
\begin{itemize}
\item Using the fact that the vectors $\xi_1, \xi_2, \dots, \xi_{m + 1}$ are pairwise independent, we can write ${(\xi_j)}_\alpha = A_j {(\xi_1)}_\alpha + B_j {(\xi_2)}_\alpha \mbox{ for } j \geq 3$  with  $A_j \neq 0 \neq B_j$ and $A_i B_j - A_j B_i \neq 0 \mbox{ for } i \neq j$. Similarly, we can write ${(\xi_j)}_\beta = A_j {(\xi_1)}_\beta + B_j {(\xi_2)}_\beta$ under same conditions on $j$.
This will convert everything in the terms of four vectors - ${(\xi_1)}_\alpha, {(\xi_2)}_\alpha, {(\xi_1)}_\beta, {(\xi_2)}_\beta$.
\item At least one of the following collections of vectors is linearly independent in $\mathbb{R}^3$. 
\begin{enumerate}
\item $\left\{ {(\xi_1)}_\alpha, {(\xi_2)}_\alpha, {(\xi_1)}_\beta \right\}$
\item $\left\{ {(\xi_1)}_\alpha, {(\xi_2)}_\alpha, {(\xi_2)}_\beta \right\}$.
\end{enumerate} 
This can be seen by the method of contradiction as follows: We know that ${(\xi_1)}_\alpha$ and ${(\xi_2)}_\alpha$ are linearly independent. If possible, assume that ${(\xi_1)}_\beta$ and ${(\xi_2)}_\beta$ both depend on ${(\xi_1)}_\alpha$ and ${(\xi_2)}_\alpha$. In other words, ${(\xi_1)}_\alpha$, ${(\xi_1)}_\beta$, ${(\xi_2)}_\alpha$ and ${(\xi_2)}_\beta$ all lie in the same plane. This implies that $\xi_1$ and $\xi_2$ both lie in the direction perpendicular to that plane generated by ${(\xi_1)}_\alpha$ and ${(\xi_2)}_\alpha$, which is a contradiction to the fact that $\xi_1$ and $\xi_2$ are independent. Hence either $\left\{{(\xi_1)}_\alpha, {(\xi_2)}_\alpha, {(\xi_1)}_\beta\right\}$ or $\left\{{(\xi_1)}_\alpha, {(\xi_2)}_\alpha, {(\xi_2)}_\beta\right\}$ are linearly independent in $\mathbb{R}^3$.\\
This step converts everything into terms of three vectors.
\item Finally, the linear independence of the given tensors follows using exactly the same arguments as in \cite[Lemma 3.4]{Microlocal_2021}. 
\end{itemize}    
\end{proof}
\noindent The above lemma generates a collection of $\binom {m + 2}m$ linearly independent $m$-tensors in $\mathbb{R}^3$ from $(m+1)$-pairwise independent vectors. Also, we know that the dimension of the space of symmetric $m$-tensors in $\mathbb{R}^3$ is $\binom {m + 2}m$. Hence any symmetric $m$-tensor $\theta^{\odot {m}}$ for $\theta \in \mathbb{R}^3 \backslash \left\{0\right\}$ can be written as follows :
\begin{align} \label{eq: theta with LI tensors} 
\begin{split}
    \theta^{\odot {m}} &= \sum_{i = 0}^m \sum_{j = 1}^{i + 1} c_{ij}(\theta) {(\xi_j)}_\alpha^{\odot i} \odot {(\xi_j)}_\beta^{\odot (m-i)} \\
    &= \sum_{i = 0}^m \sum_{j = 1}^{i + 1} c_{ij}(\theta) A_{ij},
\end{split}
\end{align}
where \[A_{ij} = {(\xi_j)}_\alpha^{\odot i} \odot {(\xi_j)}_\beta^{\odot (m-i)}.\]
Observe that any symmetric $m$-tensor in $\FR^3$ can be written as a vector with $\binom{m + 2}{2}$ components. Then using Cramer's rule for the above system of $\binom{m + 2}{2}$ linear equations, we get the coefficients as follows:
\begin{equation} \label{3}
    c_{ij}(\theta) = \frac{\Delta_{ij}(\theta)}{\Delta}
\end{equation}
where 
\[ \Delta_{ij}(\theta) = \det \hspace{1mm}\left(A_{01} \hspace{1mm} A_{11} \hspace{1mm} A_{12} \hspace{1mm} \cdots \hspace{1mm} A_{i(j-1)} \hspace{1mm} \theta^{\odot {m}} \hspace{1mm} A_{i(j+1)} \hspace{1mm} \cdots \hspace{1mm} A_{m(m+1)}\right) \]
and
\[ \Delta = \det \hspace{1mm}\left(A_{01} \hspace{1mm} A_{11} \hspace{1mm} A_{12} \hspace{1mm} \cdots \hspace{1mm} A_{i(j-1)} \hspace{1mm} A_{ij} \hspace{1mm} A_{i(j+1)} \hspace{1mm} \cdots \hspace{1mm} A_{m(m+1)}\right). \]
Also note that for $n = 3$, the condition that a collection of $(\nu(m, n) = m + 1)$-vectors are generic is equivalent to the condition that these vectors are pairwise independent \cite[Remark 1]{Microlocal_2018}. With this preparation, we are ready to start with the proof of the Proposition \ref{prop: Main proposition}, and on the way, we will prove one more important lemma.
\begin{proof}[Proof of Proposition \ref{prop: Main proposition}]
 Let $f$ and $\gamma$ be as in Theorem \ref{th:Main Theorem}. For a fixed hyperplane $H_{\omega,p}$ and a point $x \in H_{\omega,p}$, the Kirillov-Tuy condition will generate a collection of pairwise independent vectors $x - \gamma_1,\dots,x - \gamma_{m + 1}$. For $1 \leq j \leq m+1$, we can find unit vector $\xi_j$ and a scalar $t_j$ such that $t_j \xi_j = x - \gamma_j$ and the vectors in the  collection $ \displaystyle\left\{\xi_j : 1 \leq j \leq m + 1 \right\}$ are pairwise independent. Then from our assumptions, the following integrals are known:
 \begin{equation}
     \mathcal{T}^j_if(\gamma_0, \xi_j) = \int_0^\infty \left<f(\gamma_0 + t \xi_j), {(\xi_j)}_\alpha^{\odot i} \odot {(\xi_j)}_\beta^{\odot (m-i)}\right> \,dt.
 \end{equation}
 Also, relations \eqref{eq: theta with LI tensors} and \eqref{3} will give
\begin{align}\label{9}
\begin{split}
\left<f(x), \theta^{\odot m}\right> &= \sum_{i = 0}^m \sum_{j = 1}^{i + 1} \frac{\Delta_{ij}(\theta)}{\Delta} \left<f(x), {(\xi_j)}_\alpha^{\odot i} \odot {(\xi_j)}_\beta^{\odot (m-i)}\right> \\
&= \sum_{i = 0}^m \sum_{j = 1}^{i + 1} \frac{\Delta_{ij}(\theta)}{\Delta} \left<f(x), A_{ij}\right>.
\end{split}
\end{align}
Since $A_{ij}$ is completely known for a given $x$, which we can use to compute $\displaystyle \frac{\Delta_{ij}(\theta)}{\Delta}$ for $1 \leq j \leq i + 1$, $0 \leq i \leq m$, and for a given $\theta$. From equation \eqref{9}, we observe that the proof of the proposition is done if we can express $\displaystyle \left<f(x), A_{ij}\right>$ in terms of known restricted transverse ray transform. This we prove in the form of the following lemma:
\begin{lemma} \label{lemma: relation with given data}
    Let $f \mbox{ and } \gamma$ be as in Theorem \ref{th:Main Theorem}. Fix a parametrization $\gamma = \gamma(\lambda)$ for the curve $\gamma$. Further, assume $H_{\omega, p}$ be a fixed hyperplane and $x \in H_{\omega, p}$. Then for $0 \leq i \leq m$ and some unit vector $\xi$, we can write $\left<f(x), {\xi}_\alpha^{\odot i} \odot {\xi}_\beta^{\odot (m-i)}\right>$ in terms of $\mathcal{T}_if (\gamma_0, \xi)$, where $\gamma_0 \in \gamma \cap H_{\omega, p}$ and ${\xi}_\alpha$ and ${\xi}_\beta$ are orthonormal vectors associated to $\xi$.
\end{lemma}
\begin{proof}
Let us use the parametrization $x = \gamma_0 + \xi t$ for $\FR^3$ with $t \geq 0$. Then we have:
\[\mathcal{T}_if(\gamma_0, \xi) = \int_0^\infty \left<f(\gamma_0 + t \xi), {\xi}_\alpha^{\odot i} \odot {\xi}_\beta^{\odot (m-i)}\right> \,dt .\]
Let $\mathbb{S}(\omega) = \left\{\xi \in \mathbb{S}^2 : |\xi - \gamma_0| = 1 \mbox{ and } \left<\xi, \omega\right> = 0\right\}$. Apply the differential operator $\mathcal{L} = \omega_k \frac{\partial}{\partial \xi_k}$ to $\mathcal{T}_if$ and then integrate over $\mathbb{S}(\omega)$ to get:
    \begin{align} \label{eq: operator L}
        \int_{\mathbb{S} (\omega)} \mathcal{L}(\mathcal{T}_i f(\gamma_0, \xi)) \,d\omega(\xi) &= \int_{\mathbb{S} (\omega)} \omega_k \frac{\partial}{\partial \xi_k} \left(\int_0^\infty \left<f(\gamma_0 + \xi t), {\xi}_\alpha^{\odot i} \odot {\xi}_\beta^{\odot (m-i)}\right> \,dt \right)\,d\omega(\xi)\nonumber\\
        &= \int_{\mathbb{S} (\omega)} \int_0^\infty \omega_k \frac{\partial}{\partial \xi_k} \left(\left<f(\gamma_0 + \xi t), {\xi}_\alpha^{\odot i} \odot {\xi}_\beta^{\odot (m-i)}\right>\right) \,dt \,d\omega(\xi)\nonumber\\
        &= \int_{\mathbb{S} (\omega)} \int_0^\infty t \omega_k \frac{\partial}{\partial x_k} \left(\left<f(\gamma_0 + \xi t), {\xi}_\alpha^{\odot i} \odot {\xi}_\beta^{\odot (m-i)}\right>\right) \,dt \,d\omega(\xi)\nonumber\\
        &= \int_{H_{\omega, p}} \omega_k \frac{\partial}{\partial x_k} \left(\left<f(\gamma_0 + \xi t), {\xi}_\alpha^{\odot i} \odot {\xi}_\beta^{\odot (m-i)}\right>\right) \,ds\nonumber\\
        &= \frac{\partial}{\partial p} \int_{H_{\omega, p}} \left<f(\gamma_0 + \xi t), {\xi}_\alpha^{\odot i} \odot {\xi}_\beta^{\odot (m-i)}\right> \,ds\nonumber \quad \text{(using \eqref{Radon transform property})}\\
        &= \frac{\partial}{\partial p} \left[\left<f(\gamma_0 + \xi t), {\xi}_\alpha^{\odot i} \odot {\xi}_\beta^{\odot (m-i)}\right>\right]^\wedge (\omega, p).
    \end{align}
Now consider the planes parallel to $H_{\omega, p}$ and differentiate the above equation with respect to $p$ to get
    \begin{align} \label{eq: operator L 2}
    \begin{split}
        \frac{\partial}{\partial p} \int_{\mathbb{S} (\omega)} \mathcal{L}(\mathcal{T}_i f(\gamma_0, \xi)) \,d\omega(\xi) &= \frac{\partial}{\partial p} \left\{\frac{\partial}{\partial p} \left[\left<f(\gamma_0 + \xi t), {\xi}_\alpha^{\odot i} \odot {\xi}_\beta^{\odot (m-i)}\right>\right]^\wedge (\omega, p)\right\}\\
        &= \frac{\partial^2}{{\partial p}^2} \left[\left<f(\gamma_0 + \xi t), {\xi}_\alpha^{\odot i} \odot {\xi}_\beta^{\odot (m-i)}\right>\right]^\wedge (\omega, p)\\
        &\qquad + \frac{\partial}{\partial p} \left[\frac{\partial \lambda}{\partial p} \frac{\partial}{\partial \lambda}\left<f(\gamma(\lambda) + \xi t), {\xi}_\alpha^{\odot i} \odot {\xi}_\beta^{\odot (m-i)}\right>\right]^\wedge (\omega, p).
    \end{split}
    \end{align}
    Following similar steps as done above to achieve \eqref{eq: operator L} with $\mathcal{\widetilde{L}} = \frac{\partial \lambda}{\partial p} \mathcal{L} \frac{\partial}{\partial \lambda}$, we get
    \begin{equation} \label{eq: operator L tilde}
        \int_{\mathbb{S} (\omega)} \mathcal{\widetilde{L}}(\mathcal{T}_i f(\gamma_0, \xi)) \,d\omega(\xi) = \frac{\partial}{\partial p} \left[\frac{\partial \lambda}{\partial p} \frac{\partial}{\partial \lambda}\left<f(\gamma_0 + \xi t), {\xi}_\alpha^{\odot i} \odot {\xi}_\beta^{\odot (m-i)}\right>\right]^\wedge (\omega, p)
    \end{equation}
    From equations \eqref{eq: operator L 2} and \eqref{eq: operator L tilde}, we get
    \begin{align} \label{eq: final equation}
        \frac{\partial^2}{{\partial p}^2} \left[\left<f(\gamma_0 + \xi t), {\xi}_\alpha^{\odot i} \odot {\xi}_\beta^{\odot (m-i)}\right>\right]^\wedge (\omega, p) &= \frac{\partial}{\partial p} \int_{\mathbb{S} (\omega)} \mathcal{L}(\mathcal{T}_i f(\gamma_0, \xi)) \,d\omega(\xi) \nonumber\\
        &\qquad - \int_{\mathbb{S} (\omega)} \mathcal{\widetilde{L}}(\mathcal{T}_i f(\gamma_0, \xi)) \,d\omega(\xi)
    \end{align}
    To compactify the notations, let us give a notation to the right-hand side of the above equation as follows:
    $$\Wc(\Tc_i f) := \frac{\partial}{\partial p} \int_{\mathbb{S} (\omega)} \mathcal{L}(\mathcal{T}_i f(\gamma_0, \xi)) \,d\omega(\xi_j)
     - \int_{\mathbb{S} (\omega)} \mathcal{\widetilde{L}}(\mathcal{T}_i f(\gamma_0, \xi)) \,d\omega(\xi)$$
   Since $\Wc(\Tc_i f)$ is completely known from the given data, hence left-hand side of equation \eqref{eq: final equation} is also known. Finally, the Radon inversion formula gives 
    \begin{equation} \label{eq: Radon inversion}
        \left<f(x), {\xi}_\alpha^{\odot i} \odot {\xi}_\beta^{\odot (m-i)}\right> = -\frac{1}{8\pi^2} \int_{\mathbb{S}^2}\Wc(\Tc_i f)\,d\omega. 
    \end{equation}
    This completes the proof of our lemma.
\end{proof}
\noindent Now applying the above lemma to each element of collection A, we get 
\[\left<f(x), A_{ij}\right> = -\frac{1}{8\pi^2} \int_{\mathbb{S}^2}\Wc(\Tc^j_i f)\,d\omega.\]
This finishes the proof of our Proposition \ref{prop: Main proposition}.
\end{proof}
\noindent Next, we state an algebraic lemma that we will prove in the Appendix Section \ref{sec:appendix} below. This lemma is required to complete the proof of the Theorem \ref{th:Main Theorem}.
\begin{lemma} \label{algebraic lemma}
    Let $f$ be a symmetric $m$-tensor field in $\mathbb{R}^3$ and $\theta_1, \dots, \theta_m \in \mathbb{R}^3 \backslash \left\{0\right\}$. For $1 \leq k \leq m$, define 
    \[ J_k^m = \left\{(j_1, \dots, j_k) : 1 \leq j_i \leq m \mbox{ for } i = 1, \dots, k \mbox{ and } j_1 < \dots < j_k \right\}. \]
    Then we have 
    \begin{equation} \label{4}
        \left<f, (\theta_1, \dots, \theta_m)\right> = \sum_{k = 1}^m \frac{{(-1)}^{m - k}}{m!}\left( \sum_{J_k^m} \left<f, {\left(\theta_{j_1} + \dots + \theta_{j_k}\right)}^{\odot m}\right> \right).
    \end{equation}
\end{lemma}
\begin{proof}[Proof of Theorem \ref{th:Main Theorem}]
    We know that $f_{i_1 \dots i_m}(x) = \left<f(x), (e_{i_1}, \dots, e_{i_m})\right>$, where $e_i$'s are one of the three standard basis elements $(1, 0, 0), (0, 1, 0)$ and $(0, 0, 1)$. 
    Therefore, to prove Theorem \ref{th:Main Theorem}, it is sufficient to show that $\left<f, (\theta_1, \dots, \theta_m)\right>$ can be obtained explicitly in terms of given data for an arbitrary collection $\theta_1, \dots, \theta_m \in \mathbb{R}^3 \backslash \left\{0\right\}$. From Lemma \ref{algebraic lemma}, we have
    \begin{align*}
        \left<f(x), (\theta_1, \dots, \theta_m)\right> &= \sum_{k = 1}^m \frac{{(-1)}^{m - k}}{m!}\left( \sum_{J_k^m} \left<f(x), {\left(\theta_{j_1} + \dots + \theta_{j_k}\right)}^{\odot m}\right> \right)\\
        &= \sum_{k = 1}^m \frac{{(-1)}^{m - k}}{m!}\left( \sum_{J_k^m} \left<f(x), {\theta^{\odot m}_{j_1\dots j_k}}\right> \right)
    \end{align*}
    where $\theta_{j_1\dots j_k} = \theta_{j_1} + \dots + \theta_{j_k}$. Then using equations \eqref{9} and \eqref{eq: Radon inversion}, we get
    \begin{equation*}
        \left<f(x), (\theta_1, \dots, \theta_m)\right> = -\frac{1}{8\pi^2}  \sum_{k = 1}^m \frac{{(-1)}^{m - k}}{m!} \left( \sum_{J_k^m} \sum_{i = 0}^m \sum_{j = 1}^{i + 1} \frac{\Delta_{ij}(\theta_{j_1\dots j_k})}{\Delta} \int_{\mathbb{S}^2}\Wc(\Tc^j_i f)\,d\omega \right).
    \end{equation*}
We can now replace $\theta_i$'s with $e_i$'s as required to obtain all the components of the unknown tensor field. This completes the proof of Theorem \ref{th:Main Theorem}.
\end{proof}
\section{Proof of Theorem \ref{th:Main Theorem 2}} \label{sec:proof of th 2} 
Similar to the proof of Theorem \ref{th:Main Theorem}, we start with the proof of a proposition that relates $\left<f, v\right>$, for any $v \in \mathbb{R}^n \backslash \left\{0\right\}$, to the known data $\Tc_i f$.
\begin{proposition} \label{prop: Main proposition 2}
    For $1 \leq i \leq n - 1$, assume $\Tc_i f$ is known for all lines intersecting the curve $\gamma$ satisfying the conditions of Theorem \ref{th:Main Theorem 2}. Then, for any $v \in \mathbb{R}^n \backslash \left\{0\right\}$, $\left<f, v\right>$  can be written explicitly in terms of $\Tc_i f$.
\end{proposition}
\noindent The following lemma is needed to prove Proposition \ref{prop: Main proposition 2}:
\begin{lemma} \label{lemma: linear independence}
    If $\xi_1$ and $\xi_2$ are linearly independent vectors in $\mathbb{R}^n$, then there exists $l \in \left\{1, 2, \dots, n - 1\right\}$ such that the vectors ${(\eta_1)}_{1}, {(\eta_1)}_{2}, \dots, {(\eta_1)}_{(n - 1)}, {(\eta_2)}_{l}$ are linearly independent in $\mathbb{R}^n$. Here ${(\eta_i)}_{j}$ represents the orthogonal component $\eta_j$ corresponding to $\xi_i$ as defined previously (see \eqref{eq:orthogonal frame}).
\end{lemma}
\begin{proof}
    We prove this lemma using the method of contradiction. We know that ${(\eta_1)}_{1}, {(\eta_1)}_{2}, \dots,$ ${(\eta_1)}_{(n - 1)}$ are orthogonal vectors in $\mathbb{R}^n$, and hence they are linearly independent. They generate a hyperplane in $\mathbb{R}^n$, call it $H$. Now if possible, assume that all the vectors ${(\eta_2)}_{1}, {(\eta_2)}_{2}, \dots, {(\eta_2)}_{(n - 1)}$ depend on the vectors ${(\eta_1)}_{1}, {(\eta_1)}_{2}, \dots, {(\eta_1)}_{(n - 1)}$ and hence lie in $H$. Since $\xi_1$ is orthogonal to ${(\eta_1)}_{1}, {(\eta_1)}_{2}, \dots, {(\eta_1)}_{(n - 1)}$, it is perpendicular to the hyperplane $H$. For the same reason, $\xi_2$ is also perpendicular to the hyperplane $H$; hence, $\xi_1$ and $\xi_2$ are dependent, which is a contradiction. So there exists $l \in \left\{1, 2, \dots, n - 1\right\}$ such that the vectors ${(\eta_1)}_{1}, {(\eta_1)}_{2}, \dots, {(\eta_1)}_{(n - 1)}, {(\eta_2)}_{l}$ are linearly independent in $\mathbb{R}^n$.
\end{proof}
\noindent Using above lemma, we get $n$ linearly independent vectors in $\mathbb{R}^n$. Hence any vector $v \in \mathbb{R}^n \backslash \left\{0\right\}$ can be written as a linear combination of these $n$-linearly independent vectors as follows :

\begin{equation} \label{eq: v with LI vectors} 
v = c_1 {(\eta_1)}_{1} + \dots c_{n - 1} {(\eta_1)}_{(n - 1)} + c_n {(\eta_2)}_{l},
\end{equation}

\noindent Using Cramer's rule, we get the coefficients as follows :

\begin{equation} \label{eq: Cramer's rule coefficients}
    c_i(v) = \frac{\Delta_i(v)}{\Delta}
\end{equation}
where 
\[ \Delta_i(v) = \det \hspace{1mm}\left({(\eta_1)}_{1} \hspace{1mm} {(\eta_1)}_{2} \hspace{1mm} \cdots \hspace{1mm} {(\eta_1)}_{i - 1} \hspace{1mm} v \hspace{1mm} {(\eta_1)}_{i + 1} \hspace{1mm} \cdots \hspace{1mm} {(\eta_1)}_{n - 1} \hspace{1mm} {(\eta_2)}_{l}\right); \quad i = 1, \dots, n - 1, \]

\[\Delta_n(v) = \det \hspace{1mm}\left({(\eta_1)}_{1} \hspace{1mm} \cdots \hspace{1mm} {(\eta_1)}_{n - 1} \hspace{1mm} v\right)\]

\noindent and
\[ \Delta = \det \hspace{1mm}\left({(\eta_1)}_{1} \hspace{1mm} {(\eta_1)}_{2} \hspace{1mm} \cdots \hspace{1mm} {(\eta_1)}_{n - 1} \hspace{1mm} {(\eta_2)}_{l}\right). \]
\begin{proof}[Proof of Proposition \ref{prop: Main proposition 2}]
Let $f$ and $\gamma$ be as in Theorem \ref{th:Main Theorem 2}. For a fixed hyperplane $H_{\omega,p}$, the modified Kirillov-Tuy condition will generate two linearly independent vectors $x - \gamma_1$ and $x - \gamma_2$. We can find unit vectors $\xi_1, \xi_2$ and scalars $t_1, t_2$ such that $t_j \xi_j = x - \gamma_j; j = 1, 2$ and the vectors $\xi_1$ and $\xi_2$ are linearly independent. Now, relations \eqref{eq: v with LI vectors} and \eqref{eq: Cramer's rule coefficients} will give
\begin{equation}
    \left<f(x), v\right> = \sum_{i = 1}^{n - 1} \frac{\Delta_i(v)}{\Delta} \left<f(x), {(\eta_1)}_{i}\right> + \frac{\Delta_n(v)}{\Delta} \left<f(x), {(\eta_2)}_{l}\right>.
\end{equation}
In the above expression, we know $$\frac{\Delta_i(v)}{\Delta}; 1 \leq i \leq n$$ for a given $x$ and the vector $v$. Hence observe that the Proposition \ref{prop: Main proposition 2} is proved if we can express $\left<f(x), {(\eta_1)}_{i}\right>; i = 1, \dots, n - 1$ and $\left<f(x), {(\eta_2)}_{l}\right>$ in terms of the known restricted TRT. This can be done using the following lemma:
\begin{lemma} \label{lemma: relation with given data 2}
Let $f \mbox{ and } \gamma$ be as in Theorem \ref{th:Main Theorem 2}. Further, assume $H_{\omega, p}$ be a fixed hyperplane and $x \in H_{\omega, p}$. Then we can write $\left<f(x), {(\eta_1)}_{i}\right>; i = 1, \dots, n - 1$ and $\left<f(x), {(\eta_2)}_{l}\right>$ in terms of $\mathcal{T}_if, 1 \leq i \leq n - 1$.  
\end{lemma}
\noindent The proof of Lemma \ref{lemma: relation with given data 2} follows the same steps as Lemma \ref{lemma: relation with given data}, and finally we get,
\begin{equation} 
    \left<f(\gamma_0 + t \xi), {(\eta_j)}_{i}\right> = \mathcal{R}^{-1} (\Wc(\Tc^j_i f))
\end{equation}
for all $i$ and $j$ required.
\noindent This finishes the proof of proposition \ref{prop: Main proposition 2} as well.
\end{proof}
\begin{proof}[Proof of Theorem \ref{th:Main Theorem 2}]
    We know that $f_i = \left<f, e_i\right>$ where $e_i$ is one of the basis elements\\ $\left\{(1, 0, \dots, 0), (0, 1, \dots, 0), \dots, (0, 0, \dots, 1)\right\}$. Also from the previous calculation, we have  
    \begin{equation*}
        \left<f(x), v\right> = \sum_{i = 1}^{n - 1} \frac{\Delta_i(v)}{\Delta} \Rc^{-1}\left\{\Wc(\Tc^1_i f)\right\} + \frac{\Delta_n(v)}{\Delta} \Rc^{-1}\left\{\Wc(\Tc^2_l f)\right\}
    \end{equation*}
    So $f$ can be reconstructed componentwise using $e_i$'s as follows:
    \begin{equation*}
        f_i = \sum_{i = 1}^{n - 1} \frac{\Delta_i(e_i)}{\Delta} \Rc^{-1}\left\{\Wc(\Tc^1_i f)\right\} + \frac{\Delta_n(e_i)}{\Delta} \Rc^{-1}\left\{\Wc(\Tc^2_l f)\right\} ; \quad i = 1, \dots, n.
    \end{equation*}
    This completes the proof of Theorem \ref{th:Main Theorem 2}.
\end{proof}
\section{Appendix} \label{sec:appendix}
As mentioned earlier, this section is completely devoted to the proof of Lemma \ref{algebraic lemma}.
\begin{proof}[Proof of Lemma \ref{algebraic lemma}]
We know that 
\[\theta_1 \odot \dots \odot \theta_m = \frac{1}{m!} \sum_{\pi \in \Pi_m}\theta_{\pi 1} \otimes \dots \otimes \theta_{\pi m}\]
where $\Pi_m$ is the permutation group of $m$ indices and $\pi k$ represents the image of index $k$ under $\pi$.
This gives 
\[\left<f, \theta_1 \odot \dots \odot \theta_m\right> = \frac{1}{m!} \sum_{\pi \in \Pi_m} \left<f, \theta_{\pi 1} \otimes \dots \otimes \theta_{\pi m}\right>.\]
Now since $f$ is symmetric, we have 
\[\left<f, \theta_{\pi 1} \otimes \dots \otimes \theta_{\pi m}\right> = \left<f, \theta_{1} \otimes \dots \otimes \theta_{m}\right>\]
for any $\pi \in \Pi_m$. This gives 
\[\left< f, (\theta_1, \dots, \theta_m)\right> = \left<f, \theta_{1} \otimes \dots \otimes \theta_{m}\right> = \left<f, \theta_1 \odot \dots \odot \theta_m\right>\] 
We have the following properties for symmetric tensor product $\odot$:
\begin{itemize}
\item ${(\theta_1 + \theta_2)}^{\odot m} = \sum_{q = 0}^m \binom{m}{q} \theta_1^{\odot {(m - q)}} \odot \theta_2^{\odot q}$.
\item ${(\theta_1 + \dots + \theta_l)}^{\odot m} = \sum_{q_1 + \dots + q_l = m} \binom{m}{q_1, \dots, q_l} \theta_1^{\odot q_1} \odot \dots \odot \theta_l^{\odot q_l}$
\end{itemize}
where $\binom{m}{q_1, \dots, q_l} = \frac{m!}{q_1!\,\dots\,q_l!}$. \\

\noindent To obtain the required identity \eqref{4}, it is sufficient to show the following: 
    \begin{equation} \label{5}
        \theta_1 \odot \dots \odot \theta_m = \sum_{k = 1}^m \frac{{(-1)}^{m - k}}{m!}\left( \sum_{J_k^m} {\left(\theta_{j_1} + \dots + \theta_{j_k}\right)}^{\odot m} \right).
    \end{equation}
For notational convenience, we prove the following equivalent version for real numbers:
    \begin{equation} \label{6}
        v_1 \times \dots \times v_m = \sum_{k = 1}^m \frac{{(-1)}^{m - k}}{m!}\left( \sum_{J_k^m} {\left(v_{j_1} + \dots + v_{j_k}\right)}^m \right)
    \end{equation}
where $v_i$'s are real numbers. From \eqref{6}, we get \eqref{5} using the properties of the symmetric tensor product mentioned above. This result generalizes the following two special cases for general $m$.
\begin{itemize}
        \item For $m = 2$, we have 
        \[ v_1 \times v_2 = \frac{1}{2!} \left({(v_1 + v_2)}^2 - {v_1}^2 - {v_2}^2\right). \]
        \item For $m = 3$, we have 
        \[ v_1 \times v_2 \times v_3 = \frac{1}{3!} \left({(v_1 + v_2 + v_3)}^3 - {(v_1 + v_2)}^3 - {(v_1 + v_3)}^3 - {(v_2 + v_3)}^3 + {v_1}^3 + {v_2}^3 + {v_3}^3\right). \]
\end{itemize}
We will prove it using mathematical induction on $m$. From above, we see that the result is true for $m = 1, 2, 3$. Let us assume that the result is true for $m = p$, that is,
\begin{equation*}
v_1 \times \dots \times v_p = \sum_{k = 1}^p \frac{{(-1)}^{p - k}}{p!}\left( \sum_{J_k^p} {\left(v_{j_1} + \dots + v_{j_k}\right)}^p \right).
\end{equation*}
Then we will show that the result is true for $m = p + 1$. Consider the right hand side of \eqref{6} for $m = p + 1$,
\begin{align*}
I &= \sum_{k = 1}^{p + 1} \frac{{(-1)}^{p + 1 - k}}{{(p + 1)}!}\left( \sum_{J_k^{p + 1}} {\left(v_{j_1} + \dots + v_{j_k}\right)}^{p + 1} \right)\\
&= \sum_{k = 1}^{p + 1} \frac{{(-1)}^{p + 1 - k}}{{(p + 1)}!} \left( \sum_{J_{k - 1}^p} {\left(v_{j_1} + \dots + v_{j_{k - 1}} + v_{p + 1}\right)}^{p + 1} + \sum_{J_k^p} {\left(v_{j_1} + \dots + v_{j_k}\right)}^{p + 1} \right)\\
&= \sum_{k = 1}^{p + 1} \frac{{(-1)}^{p + 1 - k}}{{(p + 1)}!} \left( \sum_{J_{k - 1}^p} {\left(v_{j_1} + \dots + v_{j_{k - 1}} + v_{p + 1}\right)}^{p + 1} \right) + \sum_{k = 1}^{p + 1} \frac{{(-1)}^{p + 1 - k}}{{(p + 1)}!} \left( \sum_{J_k^p} {\left(v_{j_1} + \dots + v_{j_k}\right)}^{p + 1} \right)\\
&= I_1 + I_2
\end{align*}
where 
\[ I_1 = \sum_{k = 1}^{p + 1} \frac{{(-1)}^{p + 1 - k}}{{(p + 1)}!} \left( \sum_{J_{k - 1}^p} {\left(v_{j_1} + \dots + v_{j_{k - 1}} + v_{p + 1}\right)}^{p + 1} \right) \]
and
\[ I_2 = \sum_{k = 1}^{p + 1} \frac{{(-1)}^{p + 1 - k}}{{(p + 1)}!} \left( \sum_{J_k^p} {\left(v_{j_1} + \dots + v_{j_k}\right)}^{p + 1} \right).\]
\begin{align*}
I_1 &= \frac{{(-1)}^p}{{(p + 1)}!} v_{p + 1}^{p + 1} + \sum_{k = 2}^{p + 1} \frac{{(-1)}^{p + 1 - k}}{{(p + 1)}!} \left( \sum_{J_{k - 1}^p} {\left(v_{j_1} + \dots + v_{j_{k - 1}} + v_{p + 1}\right)}^{p + 1} \right)\\
&= \frac{{(-1)}^p}{{(p + 1)!}} v_{p + 1}^{p + 1} + \sum_{k = 1}^p \frac{{(-1)}^{p - k}}{{(p + 1)}!} \left( \sum_{J_k^p} {\left(v_{j_1} + \dots + v_{j_k} + v_{p + 1}\right)}^{p + 1} \right), \mbox{ replacing } k - 1 \mbox{ by } k \\
&= \frac{{(-1)}^p}{{(p + 1)!}} v_{p + 1}^{p + 1} + \sum_{k = 1}^p \frac{{(-1)}^{p - k}}{{(p + 1)}!} \sum_{J_k^p} \left( \sum_{q = 0}^{p + 1} \binom{p + 1}{q} {\left(v_{j_1} + \dots + v_{j_k}\right)}^{p + 1 - q} v_{p + 1}^q \right)\\
&= \frac{{(-1)}^p}{{(p + 1)!}} v_{p + 1}^{p + 1} + \sum_{k = 1}^p \frac{{(-1)}^{p - k}}{{(p + 1)}!} \sum_{J_k^p} {\left(v_{j_1} + \dots + v_{j_k}\right)}^{p + 1} + \sum_{k = 1}^p \frac{{(-1)}^{p - k}}{p!} \sum_{J_k^p} {\left(v_{j_1} + \dots + v_{j_k}\right)}^p v_{p + 1}\\
&+ \sum_{k = 1}^p \frac{{(-1)}^{p - k}}{{(p + 1)}!} \sum_{J_k^p} \left( \sum_{q = 2}^{p + 1} \binom{p + 1}{q} {\left(v_{j_1} + \dots + v_{j_k}\right)}^{p + 1 - q} v_{p + 1}^q \right) + \sum_{k = 1}^p \frac{{(-1)}^{p - k}}{{(p + 1)}!} \sum_{J_k^p} v_{p + 1}^{p + 1}\\
&= \sum_{k = 0}^p \frac{{(-1)}^{p - k}}{{(p + 1)}!} \binom{p}{k} v_{p + 1}^{p + 1} - I_2 + \sum_{k = 1}^p \frac{{(-1)}^{p - k}}{{(p + 1)}!} \sum_{J_k^p} \left( \sum_{q = 2}^{p + 1} \binom{p + 1}{q} {\left(v_{j_1} + \dots + v_{j_k}\right)}^{p + 1 - q} v_{p + 1}^q \right)\\
&+ v_1 \times \dots \times v_{p + 1}, \left(\mbox{using induction hypothesis and }  \sum_{J_k^p} 1 = \binom{p}{k}\right)\\
&= - I_2 + I_3 + v_1 \times \dots \times v_{p + 1}, \left(\mbox{using } \sum_{k = 0}^p {(-1)}^{p - k} \binom{p}{k} = 0\right)
\end{align*}
where 
\[ I_3 = \sum_{k = 1}^p \frac{{(-1)}^{p - k}}{{(p + 1)}!} \sum_{J_k^p} \left( \sum_{q = 2}^{p + 1} \binom{p + 1}{q} {\left(v_{j_1} + \dots + v_{j_k}\right)}^{p + 1 - q} v_{p + 1}^q \right). \]
Putting this value of $I_1$ into $I$, we get the following:\[I = I_3 + v_1 \times \dots \times v_{p + 1}.\]
To complete the proof of the lemma, we need to show that $I_3 = 0$. But showing $I_3 = 0$ is equivalent to showing the following:
\[\sum_{k = 1}^p {(-1)}^{p - k} \sum_{J_k^p} {\left(v_{j_1} + \dots + v_{j_k}\right)}^r = 0 \mbox{ for } r = 1, \dots, p - 1.\]
To prove this, we will need the following multinomial expansion :
\[\left(x_1 + \dots + x_l\right)^n = \sum_{q_1 + \dots + q_l = n}\binom{n}{q_1, \dots, q_l}x_1^{q_1}\dots x_l^{q_l},\]
where
\[\binom{n}{q_1, \dots, q_l} = \frac{n!}{q_1!\,\dots\,q_l!}.\]
Consider 
\begin{align}\label{7}
\begin{split}
\sum_{k = 1}^p {(-1)}^{p - k} \sum_{J_k^p} {\left(\sum_{s = 1}^k v_s\right)}^r &= {(-1)}^{(p - 1)} \left(\sum_{s = 1}^p v_s^r\right) + {(-1)}^{(p - 2)} \sum_{i_1 \neq i_2} \left(\sum_{q_1 + q_2 = r} \binom{r}{q_1, q_2} v_{i_1}^{q_1} v_{i_2}^{q_2}\right)\\
&+ {(-1)}^{(p - 3)} \sum_{i_1 \neq i_2 \neq i_3} \left(\sum_{q_1 + q_2 + q_3 = r} \binom{r}{q_1, q_2, q_3} v_{i_1}^{q_1} v_{i_2}^{q_2} v_{i_3}^{q_3} \right) + \dots\\
&+ {(-1)}^1 \sum_{i_1 \neq i_2 \neq \dots \neq i_{q - 1}} \left(\sum_{q_1 + q_2 + \dots + q_{p - 1} = r} \binom{r}{q_1, q_2, \dots, q_{p - 1}} v_{i_1}^{q_1} v_{i_2}^{q_2} \dots v_{i_{p-1}}^{q_{p-1}} \right)\\
&+ {(-1)}^0 \sum_{i_1 \neq i_2 \neq \dots \neq i_p} \left(\sum_{q_1 + q_2 + \dots + q_p = r} \binom{r}{q_1, q_2, \dots, q_p} v_{i_1}^{q_1} v_{i_2}^{q_2} \dots v_{i_p}^{q_p} \right) \\
&= T_1 + \dots + T_p,  
\end{split}
\end{align}
where
\begin{align*}
T_2 &= {(-1)}^{(p - 2)} \sum_{i_1 \neq i_2} \left(\sum_{q_1 + q_2 = r} \binom{r}{q_1, q_2} v_{i_1}^{q_1} v_{i_2}^{q_2}\right)\\
&= {(-1)}^{(p - 2)} \sum_{i_1 \neq i_2} \left(v_{i_1}^r + v_{i_2}^r + \sum_{q_1 + q_2 = r, q_1 \neq 0 \neq q_2} \binom{r}{q_1, q_2} v_{i_1}^{q_1} v_{i_2}^{q_2}\right)\\
&= (p - 1) {(-1)}^{(p - 2)} \left(\sum_{s = 1}^p v_s^r\right) + {(-1)}^{(p - 2)} \sum_{i_1 \neq i_2} \left(\sum_{q_1 + q_2 = r, q_1 \neq 0 \neq q_2} \binom{r}{q_1, q_2} v_{i_1}^{q_1} v_{i_2}^{q_2}\right),
\end{align*}
\begin{align*}
T_3 &= {(-1)}^{(p - 3)} \sum_{i_1 \neq i_2 \neq i_3} \left(\sum_{q_1 + q_2 + q_3 = r} \binom{r}{q_1, q_2, q_3} v_{i_1}^{q_1} v_{i_2}^{q_2} v_{i_3}^{q_3} \right)\\
&= {(-1)}^{(p - 3)} \sum_{i_1 \neq i_2 \neq i_3} \left(v_{i_1}^r + v_{i_2}^r + v_{i_3}^r + \sum_{q_1 + q_2 = r, q_1 \neq 0 \neq q_2} \binom{r}{q_1, q_2} \left\{v_{i_1}^{q_1} v_{i_2}^{q_2} + v_{i_1}^{q_1} v_{i_3}^{q_3} + v_{i_2}^{q_2} v_{i_3}^{q_3}\right\} \right.\\
&\qquad + \left.\sum_{q_1 + q_2 + q_3 = r, q_1 \cdot q_2 \cdot q_3 \neq 0} \binom{r}{q_1, q_2, q_3} v_{i_1}^{q_1} v_{i_2}^{q_2} v_{i_3}^{q_3}\right)\\
&= {(-1)}^{(p - 3)} \binom{p - 1}{2} \left(\sum_{s = 1}^p v_s^r\right) + (p - 1) {(-1)}^{(p - 3)} \sum_{i_1 \neq i_2} \left(\sum_{q_1 + q_2 = r, q_1 \neq 0 \neq q_2} \binom{r}{q_1, q_2} v_{i_1}^{q_1} v_{i_2}^{q_2}\right)\\
&\qquad + {(-1)}^{(p - 3)} \sum_{i_1 \neq i_2 \neq i_3} \left(\sum_{q_1 + q_2 + q_3 = r, q_1 \cdot q_2 \cdot q_3 \neq 0} \binom{r}{q_1, q_2, q_3} v_{i_1}^{q_1} v_{i_2}^{q_2} v_{i_3}^{q_3} \right).
\end{align*}
    Similarly, we can expand every $T_k$ for $k = 4, \dots, p$. Putting it back to \eqref{7}, we get
    \begin{align*}
        \sum_{k = 1}^p {(-1)}^{p - k} \sum_{J_k^p} {\left(\sum_{s = 1}^k v_s\right)}^r &= \left(\sum_{s = 1}^p v_s^r\right) \sum_{k = 0}^{p - 1} \binom{p - 1}{k} {(-1)}^{(p - 1 - k)}\\
        &+ \sum_{i_1 \neq i_2} \left(\sum_{q_1 + q_2 = r, q_1 \neq 0 \neq q_2} \binom{r}{q_1, q_2} v_{i_1}^{q_1} v_{i_2}^{q_2}\right) \times \sum_{k = 0}^{p - 2} \binom{p - 2}{k} {(-1)}^{(p - 2 - k)} + \cdots\\
        &+ \sum_{i_1 \neq i_2 \neq \dots \neq i_{p - 1}}\left(\sum_{q_1 + \dots + q_{p - 1} = r, \prod_{l = 1}^{p - 1}q_l \neq 0} \binom{r}{q_1, q_2, \dots, q_{p - 1}} v_{i_1}^{q_1} \dots v_{i_{p - 1}}^{q_{p - 1}} \right.\\
        &\qquad \left.\times \sum_{k = 0}^1 \binom{1}{k} {(-1)}^{1 - k}\right)\\
        &+ \sum_{i_1 \neq i_2 \neq \dots \neq i_p}\left(\sum_{q_1 + \dots + q_p = r, \prod_{l = 1}^p q_l \neq 0} \binom{r}{q_1, q_2, \dots, q_p} v_{i_1}^{q_1} \dots v_{i_p}^{q_p}\right)\\
        &= 0,
    \end{align*}
where the last equality follows from the fact that every term except the last one has a factor of the form ${(1 - l)}^l$ and the last term is $0$ because $r \leq (p - 1)$.

Using this equality, we have $I_3 = 0$, so our result is also true for $m = p + 1$, which completes the proof of our lemma by mathematical induction.
\end{proof}

\section{Acknowledgements}\label{sec:acknowledge}
 RM was partially supported by SERB SRG grant No. SRG/2022/000947. CT was supported by the PMRF fellowship from the Government of India.
\bibliographystyle{amsplain}
\bibliography{reference}

\end{document}